\documentclass[12pt]{amsart}
\usepackage{amsmath, amsthm, amssymb, graphicx, color, xcolor, array}
\usepackage{hyperref}
\usepackage{tikz}
\usepackage{url}
\usepackage{ifpdf}
\usepackage[config, labelfont={normalsize}]{caption,subfig}

\newtheorem{thm}{Theorem}

\newtheorem{lem}[thm]{Lemma}

\newcommand{\prob}{\mathbb{P}}

\ifpdf
\newcommand{\includegraphicsA}[2][1]{\includegraphics[#1]{#2.pdf}}
\else
\newcommand{\includegraphicsA}[2][1]{\includegraphics[#1]{#2.eps}}
\fi

\newcommand{\symgroup}{\mathfrak{S}}
\newcommand{\bumping}{\textnormal{\textsf{B}}}
\title[Limit shapes of bumping routes]{Limit shapes of bumping routes \\ in the Robinson-Schensted correspondence}

\author{Dan Romik}
\address{Department of Mathematics,
University of California, Davis,
One Shields Avenue, Davis, CA 95616, USA}
\email{romik@math.ucdavis.edu}

\author{Piotr \'Sniady}
\address{Zentrum Mathematik, M5,
Technische Universit\"at M\"unchen, \linebreak
Boltzmannstrasse 3,
85748 Garching, Germany \newline \indent
Instytut Matematyczny, Polska Akademia Nauk, \linebreak
\mbox{ul.~\'Sniadec\-kich 8,} 00-956 Warszawa, Poland
 \newline
\indent 
Instytut Matematyczny,
Uniwersytet Wroc\l{}awski,  \mbox{pl.~Grunwaldzki~2/4,} 50-384
Wroclaw, Poland
} 
\email{piotr.sniady@tum.de, piotr.sniady@math.uni.wroc.pl}



\begin{document}

\begin{abstract}
We prove a limit shape theorem describing the asymptotic shape of bumping routes when the Robinson-Schensted algorithm is applied to a finite sequence of independent, identically distributed random variables  
with the uniform distribution $U[0,1]$ on the unit interval, followed by an insertion of a deterministic number $\alpha$. 
The bumping route converges after scaling, in the limit as the length of the sequence tends to infinity, to an
explicit, deterministic curve depending only on $\alpha$. 
This extends our previous result on the asymptotic determinism of Robinson-Schensted insertion, and answers a question posed by Moore in 2006.
\end{abstract}

\maketitle

\section{Introduction}

Let $\symgroup_n$ denote the symmetric group of order $n$. Recall that the \textbf{Robinson-Schensted correspondence} associates with a permutation $\sigma^{(n)}\in \symgroup_n$ a pair of standard Young tableaux $(P_n,Q_n)$ whose common shape $\lambda$ is a Young diagram of order $n$. 
A fruitful area of study concerns asymptotic properties of the Robinson-Schensted shape $\lambda$ and the tableaux $P_n,Q_n$ associated with a \emph{random} permutation $\sigma^{(n)}$ sampled from the uniform distribution on $\symgroup_n$. The existing results on this subject are too numerous to list here, but some of the important highlights of the theory are the limit shape result of Logan-Shepp \cite{loganshepp} and Vershik-Kerov \cite{vershikkerov1, vershikkerov2}, which led to the solution of the so-called Ulam-Hammersley problem on the typical length of a longest increasing subsequence in random permutations; and the celebrated Baik-Deift-Johansson theorem \cite{baik-deift-johansson} and its refinements and variants \cite{baik-deift-johansson2, borodin-okounkov-olshanski, johansson} that tied the behavior of longest increasing subsequences in random permutations to the Tracy-Widom distribution and other naturally-occurring stochastic processes from random matrix theory. See the book \cite{romik} for a survey of many of these developments that also touches on diverse connections to random growth processes, interacting particle systems, representation theory and more.

In this paper we continue this line of investigation by studying the \textbf{bumping route} computed during the application of an insertion procedure, which is the fundamental building block of the Robinson-Schen\-sted correspondence.
Let us recall briefly the relevant definitions. A \textbf{Young diagram} $\lambda$ of order~$n$ is an integer partition of $n$, that is, an array of positive integers $\lambda_1\ge\dots\ge\lambda_m\ge 0$ such that $n=\sum_{j=1}^m \lambda_j$, represented graphically as a diagram of left-justified square boxes wherein the $j$th row contains $\lambda_j$ boxes.
If $\lambda$ is a Young diagram of order $n$ and $x_1,\dots,x_n$ are distinct real numbers, an \textbf{increasing tableau} of shape $\lambda$ and entries given by $x_1,\dots, x_n$ is a filling of the boxes of $\lambda$ with the numbers $x_1,\dots,x_n$ that is increasing along rows and columns. A \textbf{standard Young tableau} is such an increasing tableau whose entries are precisely the numbers $1,\dots,n$.

Given an increasing tableau $P$ with entries $x_1,\dots, x_n$ and a number $z$ distinct from $x_1,\dots,x_n$, the \textbf{insertion procedure} applied to $P$ and $z$ produces a new increasing tableau $P\leftarrow z$ with entries $x_1,\dots,x_n,z$ whose shape $\lambda^+$ is obtained from $\lambda$ by the addition of a single box. 
The new tableau $P\leftarrow z$ is computed by performing a succession of \textbf{bumping steps} whereby the number $z$ is inserted into the first row of the diagram, displacing an existing entry from the first row;
the displaced entry is bumped onto the second row, and in turn bumps an entry of the second row onto the third row; and so on, until finally the entry being bumped settles down in an unoccupied position outside the diagram $\lambda$. In each row, the position where the bumping (or settling-down, in the last step) occurs is the leftmost position containing an entry bigger than the incoming number, or a new unoccupied position to the right of all existing entries if no such entry exists. 
An example is shown in Figure~\ref{fig:bumping-example}. 

\newcommand{\mycolor}{blue!50}
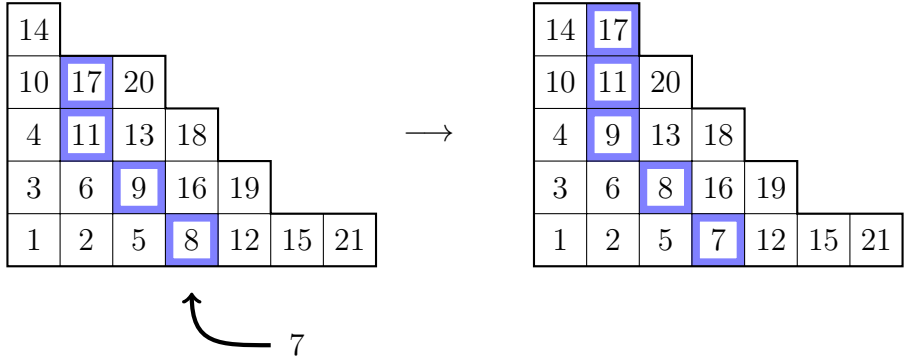
\begin{figure}[t]
\begin{center}
\begin{tikzpicture}[scale=0.7]
\begin{scope}
      \begin{scope} \clip (3,0) rectangle +(1,1); \draw[line width=6pt,\mycolor] (3,0) rectangle +(1,1); \end{scope}
      \begin{scope} \clip (2,1) rectangle +(1,1); \draw[line width=6pt,\mycolor] (2,1) rectangle +(1,1); \end{scope}
      \begin{scope} \clip (1,2) rectangle +(1,1); \draw[line width=6pt,\mycolor] (1,2) rectangle +(1,1); \end{scope}
      \begin{scope} \clip (1,3) rectangle +(1,1); \draw[line width=6pt,\mycolor] (1,3) rectangle +(1,1); \end{scope}
      \draw (0.5,0.5) node {1};
      \draw (1.5,0.5) node {2};
      \draw (2.5,0.5) node {5};
      \draw (3.5,0.5) node {8};
      \draw (4.5,0.5) node {12};
      \draw (5.5,0.5) node {15};
      \draw (6.5,0.5) node {21};
      \draw (0.5,1.5) node {3};
      \draw (1.5,1.5) node {6};
      \draw (2.5,1.5) node {9};
      \draw (3.5,1.5) node {16};
      \draw (4.5,1.5) node {19};
      \draw (0.5,2.5) node {4};
      \draw (1.5,2.5) node {11};
      \draw (2.5,2.5) node {13};
      \draw (3.5,2.5) node {18};
      \draw (0.5,3.5) node {10};
      \draw (1.5,3.5) node {17};
      \draw (2.5,3.5) node {20};
      \draw (0.5,4.5) node {14};
      \draw[thick](0,0) -- (7,0) -- (7,1) -- (5,1) -- (5,2) -- (4,2) -- (4,3) -- (3,3) -- (3,4) -- (1,4) -- (1,5) -- (0,5) -- cycle;
      \begin{scope}
      \clip (0,0) -- (7,0) -- (7,1) -- (5,1) -- (5,2) -- (4,2) -- (4,3) -- (3,3) -- (3,4) -- (1,4) -- (1,5) -- (0,5) -- cycle;
      \draw (0,0) grid (50,50);
      \end{scope}
    \draw (5.5,-1.5) node {7};
    \draw[->,ultra thick] (5,-1.5) .. controls (4,-1.5) and (3.5,-1.5) .. (3.5,-0.5);
\end{scope}
\draw (8,2.5) node {$\longrightarrow$};
\begin{scope}[shift={(10,0)}]
    \begin{scope} \clip (3,0) rectangle +(1,1); \draw[line width=6pt,\mycolor] (3,0) rectangle +(1,1); \end{scope}
    \begin{scope} \clip (2,1) rectangle +(1,1); \draw[line width=6pt,\mycolor] (2,1) rectangle +(1,1); \end{scope}
    \begin{scope} \clip (1,2) rectangle +(1,1); \draw[line width=6pt,\mycolor] (1,2) rectangle +(1,1); \end{scope}
    \begin{scope} \clip (1,3) rectangle +(1,1); \draw[line width=6pt,\mycolor] (1,3) rectangle +(1,1); \end{scope}
    \begin{scope} \clip (1,4) rectangle +(1,1); \draw[line width=6pt,\mycolor] (1,4) rectangle +(1,1); \end{scope}
    \draw (0.5,0.5) node {1};
    \draw (1.5,0.5) node {2};
    \draw (2.5,0.5) node {5};
    \draw (3.5,0.5) node {7};
    \draw (4.5,0.5) node {12};
    \draw (5.5,0.5) node {15};
    \draw (6.5,0.5) node {21};
    \draw (0.5,1.5) node {3};
    \draw (1.5,1.5) node {6};
    \draw (2.5,1.5) node {8};
    \draw (3.5,1.5) node {16};
    \draw (4.5,1.5) node {19};
    \draw (0.5,2.5) node {4};
    \draw (1.5,2.5) node {9};
    \draw (2.5,2.5) node {13};
    \draw (3.5,2.5) node {18};
    \draw (0.5,3.5) node {10};
    \draw (1.5,3.5) node {11};
    \draw (2.5,3.5) node {20};
    \draw (0.5,4.5) node {14};
    \draw (1.5,4.5) node {17};
    \draw[thick](0,0) -- (7,0) -- (7,1) -- (5,1) -- (5,2) -- (4,2) -- (4,3) -- (3,3) -- (3,4) -- (2,4) -- (2,5) -- (0,5) -- cycle;
    \begin{scope}
    \clip (0,0) -- (7,0) -- (7,1) -- (5,1) -- (5,2) -- (4,2) -- (4,3) -- (3,3) -- (3,4) -- (2,4) -- (2,5) -- (0,5) -- cycle;
    \draw (0,0) grid (50,50);
    \end{scope}
\end{scope} 
\end{tikzpicture}

\caption{Inserting a number into a tableau results in a cascade of bumping events. The bumping route is the sequence of positions where a bumping occurred.}
\label{fig:bumping-example}
\end{center}
\end{figure}

Define the \textbf{bumping route} $\bumping_{P,z}$ associated with an insertion procedure performed on the tableau $P$ with a new input~$z$ to be the sequence of positions where a bumping occurred during the insertion, together with the position of the final box added to the shape. The $j$th position in the bumping route is of the form $\big(b_{P,z}(j),j\big)$, so it is convenient to encode the bumping route using only the $x$-coordinates of the positions, which form a monotone nonincreasing sequence of positive integers $b_{P,z}(1) \ge b_{P,z}(2) \ge \dots \ge b_{P,z}(k_{P,z})$ whose length we denote by $k_{P,z}$. For example, the bumping route associated with the insertion step in Figure~\ref{fig:bumping-example} is $(4,3,2,2,2)$.

The \textbf{insertion tableau} corresponding to a sequence $x_1,\dots,x_n$ is defined as 
the outcome of the iterative application of the insertion procedure
\[ P(x_1,\dots,x_n):=\Big( \big( (\emptyset \leftarrow x_1) \leftarrow x_2 \big) \leftarrow \dots \Big) \leftarrow x_n,\]
starting with the empty tableau $\emptyset$.
The \textbf{Robinson-Schensted correspondence} associates with a permutation $\sigma^{(n)}\in \symgroup_n$ a pair of standard Young tableaux $(P_n,Q_n)$ with the first one $P_n=P\big(\sigma^{(n)})=P\big(\sigma^{(n)}(1),\dots,\sigma^{(n)}(n)\big)$ being the insertion tableau corresponding to the permutation; the definition of the second tableau $Q_n$, known as the \textbf{recording tableau}, will not be necessary for our purposes. More details on Young tableaux, the Robinson-Schensted correspondence and their properties can be found in several well-known sources such as \cite{fulton, knuthvol2, stanleyvol2}.

If $\sigma^{(n)}$ is a uniformly random permutation of order $n$, the bumping route computed in the last insertion step 
performed while calculating the insertion tableau $P_n$
is equal to $\bumping_{P_{n-1},\sigma^{(n)}(n)}$, where 
$P_{n-1}=P\left(\sigma^{(n)}(1),\dots,\sigma^{(n)}(n-1)\right)$ denotes the insertion tableau computed from the truncated sequence. The question we wish to address is that of understanding the asymptotic behavior of this bumping route.

Since the computation depends only on the relative order of the numbers $\sigma^{(n)}(j), 1\le j\le n$, it will be equivalent, and more convenient, to formulate the result in terms of a sequence $X_1,\dots,X_n$ of independent and identically distributed (i.i.d.) random variables with the uniform distribution $U[0,1]$ on the unit interval $[0,1]$, which gives a canonical way of realizing uniformly random order structures of all orders $n$ simultaneously. If we denote by $T_{n}=P(X_1,\dots,X_n)$ the corresponding insertion tableau, then the bumping route $\bumping_{T_{n-1},X_n}$ is equal in distribution to $\bumping_{P_{n-1},\sigma^{(n)}(n)}$.

Our result will pertain to an even more general scenario in which the final input $X_n$ is taken to be an arbitrary (non-random) number $\alpha$ in the unit interval $[0,1]$. Note that, by obvious monotonicity properties of the insertion procedure, as $\alpha$ increases from $0$ to $1$, the bumping route is deformed monotonically between the two extreme cases $\alpha=0$ and $\alpha=1$, where in the case $\alpha=0$ the bumping route will be the first column of the diagram and an additional new box at the top of the first column, and in the case $\alpha=1$ the bumping route consists of a single new box at the end of the first row of the diagram. 
Note also that the bumping route (except for the last box)
is contained in the Young diagram of the tableau $T_{n-1}$; this random Young diagram, whose distribution is known as the \textbf{Plancherel measure} of order $n-1$, converges to a well-known limit shape discovered in the celebrated works of Logan-Shepp \cite{loganshepp} and Vershik-Kerov \cite{vershikkerov1, vershikkerov2}.

Figure~\ref{fig:simulation} shows the bumping routes $\bumping_{T_{n-1},\alpha}$ for various values of $\alpha$ in a numerical simulation with $n=10^4$. Our goal will be to show that the bumping routes converge after scaling to a family of deterministic limiting curves, which are shown in Figure~\ref{fig:theory}. As preparation for the precise formulation of this result, let us first define this family of curves. First, define auxiliary functions $F$, $\Omega$, $u_\alpha$, $v_\alpha$, $x_\alpha$, $y_\alpha$, $\kappa$ by

\begin{figure}[t]
\hfill
\subfloat[][]{\includegraphicsA[width=0.45\textwidth]{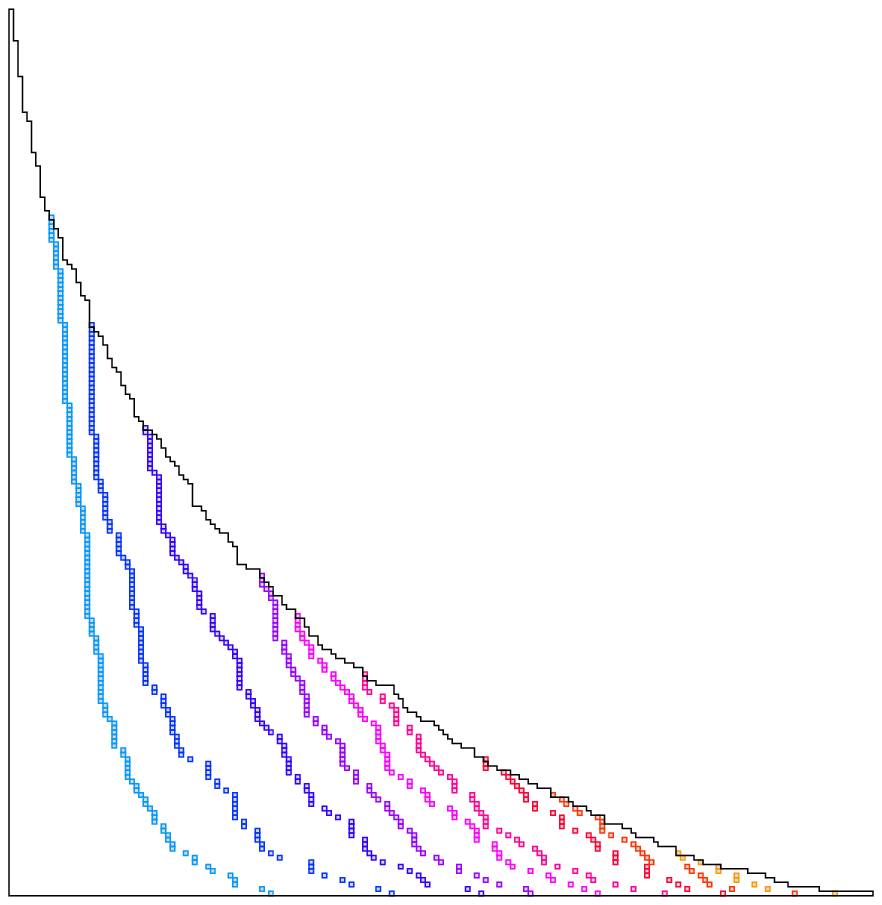} \label{fig:simulation}}
\hfill
\subfloat[][]{\includegraphicsA[width=0.45\textwidth]{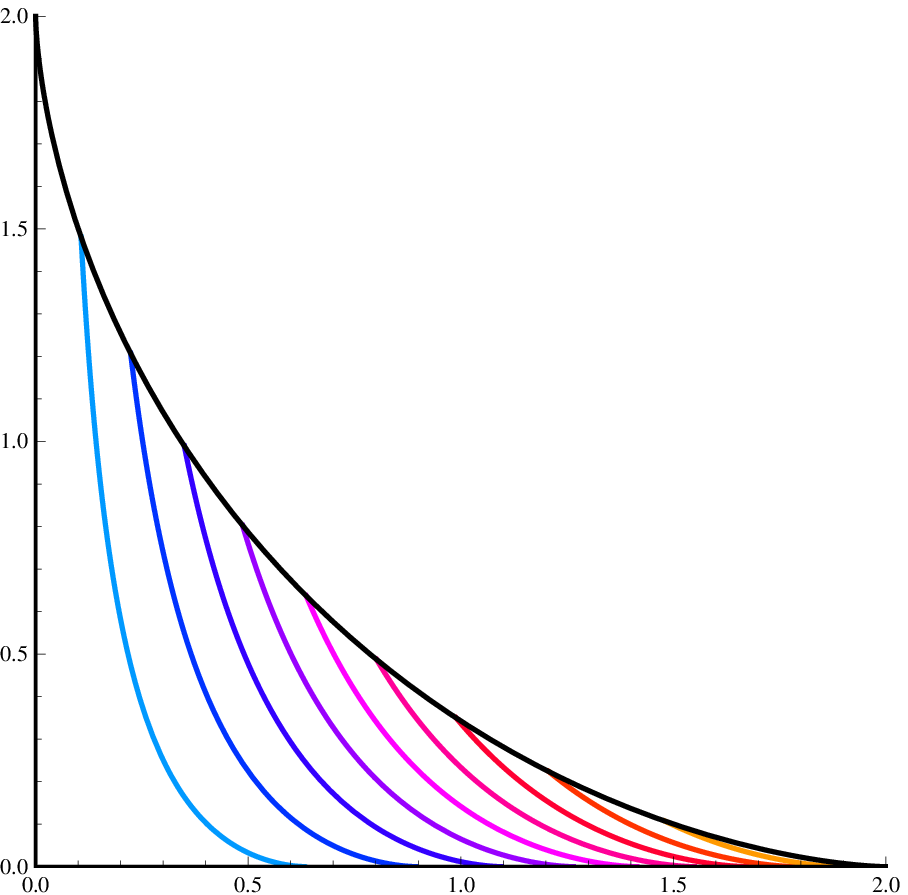} \label{fig:theory}}
\hfill

\caption{\protect\subref{fig:simulation} Bumping routes $\bumping_{T_{n-1},\alpha}$ in a simulation with $n=10^4$ and the values $\alpha=1/10,2/10,\dots,9/10$; 
\protect\subref{fig:theory} the limiting curves $\big(\beta_\alpha(s),s\big)$ for the same values of $\alpha$. The region bounding the limit shapes is the Logan-Shepp-Vershik-Kerov limit shape of Plancherel-random Young diagrams \cite{romik}.}
\label{fig:bumping-routes}
\end{figure}

\begin{align}
\Omega(u) &= \frac{2}{\pi} \left( u \sin^{-1}\left(\frac{u}{2}\right) + \sqrt{4-u^2} \right) & (|u|\le 2), \label{eq:def-omega} \\ 
F(u) &= \frac12 + \frac{1}{\pi} \left( \frac{u \sqrt{4-u^2}}{4} + \sin^{-1}\left( \frac{u}{2}\right) \right) & (|u|\le 2), \label{eq:def-semicircle-cdf} \\
u_\alpha(t) &= \sqrt{t} \, F^{-1}\left(\frac{\alpha}{t}\right) & (0\le \alpha \le t \le 1), \label{eq:ualpha} 
\\ 
v_\alpha(t) &= \sqrt{t} \,\Omega\!\left( F^{-1}\left(\frac{\alpha}{t}\right) \right) & (0\le \alpha \le t \le 1),  \label{eq:valpha} \\ 
x_\alpha(t) &= \frac{v_\alpha(t)+u_\alpha(t)}{2} & (0\le \alpha \le t \le 1),  
\label{eq:xalpha} 
\\
y_\alpha(t) &= \frac{v_\alpha(t)-u_\alpha(t)}{2} & (0\le \alpha \le t \le 1), 
\label{eq:yalpha} \\
\kappa(\alpha) &= y_\alpha(1) = \frac{\Omega(F^{-1}(\alpha)) - F^{-1}(\alpha)}{2} & (0\le \alpha \le 1).
\end{align}

The \textbf{limiting bumping route curves} are now defined as the one-parameter family $\big(\beta_\alpha(t)\big)_{0\le \alpha< 1}$ of functions, where for each $\alpha\in [0,1)$, $\beta_\alpha(\cdot)$ is given by
\begin{equation} 
\beta_\alpha(s) = x_\alpha ( y_\alpha^{-1}(s)) \qquad (0\le s\le \kappa(\alpha)). \label{eq:def-beta}
\end{equation}

Our main result is as follows.

\begin{thm}[Limit shapes of bumping routes]
\label{thm:main-result}
For each $0\le \alpha<1$, the curve $\beta_\alpha(\cdot)$ describes the limiting bumping route $\bumping_{T_{n-1},\alpha}$, in the following precise sense: for any $\epsilon>0$, we have that
\begin{align}
\prob &\left(
\left| \frac{k_{T_{n-1},\alpha}}{\sqrt{n}} - \kappa(\alpha) \right| > \epsilon
\right) \xrightarrow[n\to\infty]{} 0, \textrm{ and}
\label{eq:limit1} \\[3pt]
\prob & \left(
\max_{1 \le m\le k_{T_{n-1},\alpha}} 
\left| \frac{b_{T_{n-1},\alpha}(m)}{\sqrt{n}} -
\beta_\alpha\left( \frac{m}{\sqrt{n}} \wedge \kappa(\alpha) \right)
\right| > \epsilon
\right) \xrightarrow[n\to\infty]{} 0.
\label{eq:limit2}
\end{align}
\end{thm}

The problem of understanding the limit shapes of the bumping routes, which our result answers, was posted by C.~Moore on his personal web page (along with simulation results similar to our Figure~\ref{fig:simulation}) in 2006 \cite{moore}. 

\section{Preliminary remarks}

As a first step towards proving Theorem~\ref{thm:main-result}, let us recall some facts from the theory of Plancherel measure, which will immediately prove \eqref{eq:limit1} and also help elucidate the somewhat involved definition of the family of limiting curves $\beta_\alpha(\cdot)$. First, when discussing Plancherel-random Young diagrams and Young tableaux it is convenient to use the so-called rotated (also known as ``Russian'') coordinate system, related to the standard $x$-$y$ coordinates by the linear change of variables
\begin{equation} \label{eq:russian}
\begin{array}{c} u = x-y, \\[5pt] v=x+y. \end{array} 
\end{equation}
In this coordinate system, the curve $v=\Omega(u)$, where $\Omega(u)$ is defined in~\eqref{eq:def-omega}, describes the Logan-Shepp-Vershik-Kerov limit shape of Plancherel-random Young diagrams mentioned in the introduction.

Second, the function $F(u)$ defined in \eqref{eq:def-semicircle-cdf} is the cumulative distribution function of the semicircle distribution on $[-2,2]$; that is, we have
$$ F(u) = \frac{1}{2\pi} \int_{-2}^u \sqrt{4-s^2}\,ds \qquad (|u|\le 2). $$
Its importance for the present discussion is that, according to one of the main results of our previous paper \cite{romik-sniady}, the point 
\begin{equation} \label{eq:u-v-alpha-def}
\big(U(\alpha),V(\alpha)\big)=\left(F^{-1}(\alpha), \Omega\big(F^{-1}(\alpha)\big)\right)
\end{equation}
is the limiting scaled position (in rotated coordinates) of the new box added to the Robinson-Schensted shape after applying the insertion procedure with the number $\alpha \in[0,1]$ to the existing insertion tableau $T_{n-1}$. More precisely, when stated using our current terminology, the result \cite[Theorem~5.1]{romik-sniady}, which we dubbed the ``\textbf{asymptotic determinism of RSK insertion},'' says that the last position $\big(b_\alpha(k_{T_{n-1,\alpha}}),k_{T_{n-1,\alpha}}\big)$ of the bumping route $\bumping_{T_{n-1},\alpha}$ satisfies
\begin{multline} 
\label{eq:asym-det}
\frac{1}{\sqrt{n}} \big(b_\alpha(k_{T_{n-1,\alpha}})-k_{T_{n-1,\alpha}}, b_\alpha(k_{T_{n-1,\alpha}})+k_{T_{n-1,\alpha}}\big)
\\  \xrightarrow[n\to\infty]{\prob} \big(U(\alpha),V(\alpha)\big).
\end{multline}
After applying the inverse transformation of \eqref{eq:russian} to rewrite the result in $x$-$y$ coordinates, and noting that 
$$\kappa(\alpha)=y_\alpha(1)= \frac{V(\alpha)-U(\alpha)}{2},$$
we get the relation \eqref{eq:limit1}, the first claim of Theorem~\ref{thm:main-result}.

Next, turn to the pair of functions $\big(u_\alpha(t), v_\alpha(t)\big)$ defined in \eqref{eq:ualpha}--\eqref{eq:valpha}, with the associated pair $\big(x_\alpha(t),y_\alpha(t)\big)$ from \eqref{eq:xalpha}--\eqref{eq:yalpha} representing the same functions in $x$-$y$ coordinates. Note that for any fixed $\alpha$, the planar curve $\big(x_\alpha(t),y_\alpha(t)\big)_{\alpha\le t\le 1}$ is a reparametrized version of the curve $\big(\beta_\alpha(s),s\big)_{0\le s\le \kappa(\alpha)}$, which according to our claim \eqref{eq:limit2} is the limit shape of the bumping route $\bumping_{T_{n-1},\alpha}$ (one needs to note that $y_\alpha(\cdot)$ is a strictly decreasing function; see Lemma~\ref{lem:decreasing} below). 
It turns out that the parametrization of the curve as $\big(x_\alpha(t),y_\alpha(t)\big)$ is the correct one when trying to prove the limit shape result (although the parametrization $\big(\beta_\alpha(s),s\big)$ is the one that answers the original question). To see why, we need to explain the role of the parameter $t$. Note that for fixed $t$, the points $\big(u_\alpha(t), v_\alpha(t)\big)_{0\leq \alpha\leq 1}$ all lie on the curve $v=\sqrt{t} \, \Omega(u/\sqrt{t})$, which is a copy of the limit shape $v=\Omega(u)$ scaled down by the factor $\sqrt{t}$. The idea is that this scaled-down copy represents the limiting shape of the ``$t$-sublevel tableau'' of $T_{n-1}$---that is, the subset of boxes of $T_{n-1}$ containing an entry $\le t$. We will show below that the point $\big(u_\alpha(t),v_\alpha(t)\big)$ (or $\big(x_\alpha(t),y_\alpha(t)\big)$, in the usual coordinate system) corresponds to the limiting position, after scaling, of the point at which the bumping route $\bumping_{T_{n-1},\alpha}$ exits this sublevel tableau. The reason for this is that this exit position relates to the sublevel tableau in roughly the same way that the final bumping route position $\big(b_\alpha(k_{T_{n-1,\alpha}}),k_{T_{n-1,\alpha}}\big)$ relates to the entire tableau $T_{n-1}$, except that there is a scaling relation that causes the number being inserted to change from $\alpha$ to $\alpha/t$. A more precise formulation of this statement is discussed in the next section, after which we will see that our main result follows without much difficulty by another appeal to the ``asymptotic determinism of RSK'' theorem. A schematic illustration of the argument described above is shown in Figure~\ref{fig:summary-picture}.

\begin{figure}[t]
\begin{center}
\begin{picture}(320,300)(0,0)
\put(0,0){\scalebox{1.2}{\includegraphicsA[]{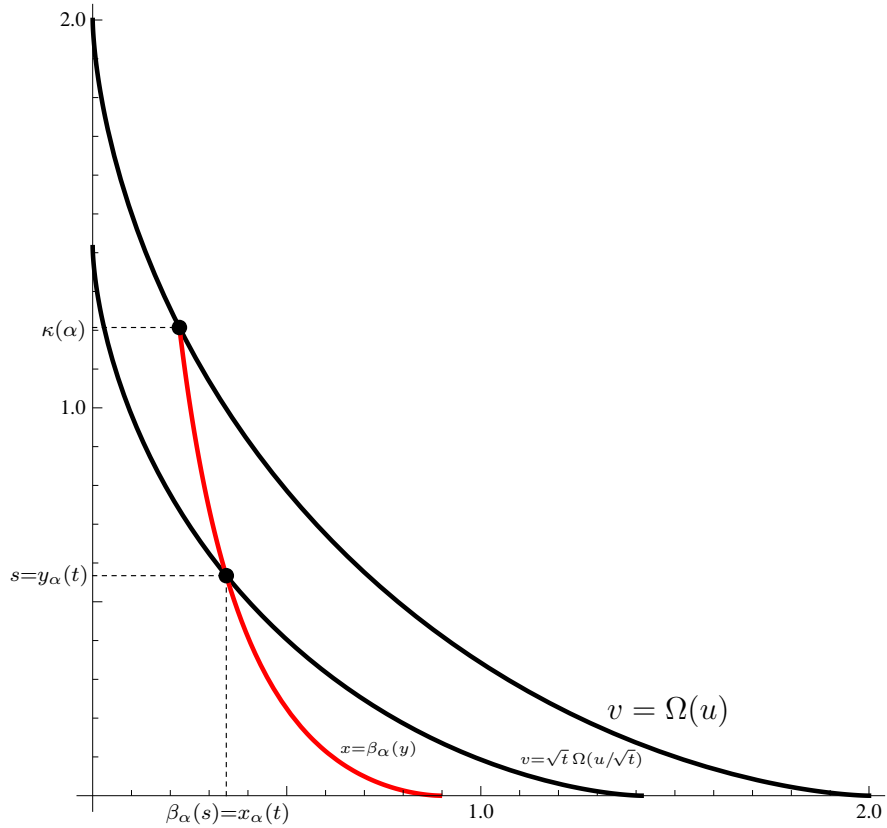}}}
\put(207,39){$v=\Omega(u)$}
\put(174,22){$\scriptscriptstyle v=\sqrt{t}\,\Omega(u/\!\sqrt{t})$}
\put(106,26){$\scriptscriptstyle x=\beta_\alpha(y)$}
\put(40,1){$\scriptstyle \beta_\alpha(s)=x_\alpha(t)$}
\put(-19,91){$\scriptstyle s=y_\alpha(t)$}
\put(-7,184){$\scriptstyle \kappa(\alpha)$}
\end{picture}
\caption{The meaning of the parameter $t$: for a fixed value of $\alpha$, the intersection $\big(u_\alpha(t),v_\alpha(t)\big)$ of the asymptotic bumping curve with the scaled-down copy  $v=\sqrt{t}\,\Omega(u/\sqrt{t})$ of the Logan-Shepp-Vershik-Kerov limit shape is computed by applying the asymptotic determinism theorem to the $t$-sublevel tableau; when the numbers in the sublevel tableau are scaled to the range $[0,1]$, the number $\alpha$ being inserted is scaled to~$\alpha/t$.}
\label{fig:summary-picture}
\end{center}
\end{figure}

We conclude this section with another small but useful observation.

\begin{lem} \label{lem:decreasing}
The function $y_\alpha(\cdot)$ is strictly decreasing. In particular, the limit shape functions $\beta_\alpha(\cdot)$ defined in \eqref{eq:def-beta} are well-defined.
\end{lem}

\begin{proof}
Denote 
$$\big(X(\alpha),Y(\alpha)\big)=\frac12\big(V(\alpha)+U(\alpha), V(\alpha)-U(\alpha)\big),$$ 
where $U(\alpha),V(\alpha)$ are defined in \eqref{eq:u-v-alpha-def}. First, note that $Y(\cdot)$ is strictly decreasing, since it is the composition of the increasing function $\alpha\mapsto F^{-1}(\alpha)$ with the function $u\mapsto \frac12(\Omega(u)-u)$, the latter being decreasing (this can be seen by interpreting this function geometrically, or by differentiating and using the fact that $\Omega'(u)=\frac{2}{\pi}\sin^{-1}(u/2)$). Now, 
if $t<t'$ are numbers in $[\alpha,1]$ then
$$
y_\alpha(t) = \sqrt{t}\, Y(\alpha/t) < \sqrt{t'}\, Y(\alpha/t) < \sqrt{t'}\, Y(\alpha/t') = y_\alpha(t'),
$$
proving the claim.
\end{proof}

\section{Plancherel measure and sublevel tableaux}

Recall that the distribution of the insertion tableau $P_n$ associated via the Robinson-Schensted correspondence with a uniformly random permutation $\sigma^{(n)}$ in $\symgroup_n$ is also called the Plancherel measure of order $n$ (this refers to the measure on standard Young tableaux; the distribution of the shape of this tableau is the Plancherel measure on Young diagrams of order $n$ that was mentioned in the introduction). By the remarks made in the introduction, a tableau $P_n$ with this distribution arises by taking the tableau $T_n$ and ``standardizing'' it by replacing each entry $X_k$ with the ranking of $X_k$ in the list $X_1,\dots,X_n$ (i.e., the number $\sigma(k)$ such that $X_k = X^{(n)}_{\sigma(k)}$, where $X^{(n)}_1\le \dots \le X^{(n)}_n$ are the order statistics of $X_1,\dots,X_n$). 
Note also that the vector 
$(X^{(n)}_1, \dots , X^{(n)}_n)$
of order statistics is independent of the tableau $P_n$ and is distributed uniformly on the simplex $$\Delta_n=\{ (x_1,\dots,x_n)\,:\, 0\le x_1\le \dots \le x_n\le 1\}.$$

It follows that, conversely, if we start with a random standard Young tableau $P_n$ distributed according to the Plancherel measure of order $n$ and a random vector $(W_1,\dots,W_n)$ that is independent of $P_n$ and is distributed uniformly on $\Delta_n$, then the tableau $T_n'$ obtained by replacing each entry $p$ of $P_n$ by $W_p$ is equal in distribution to $T_n$.

We will now apply these observations to prove a simple lemma about sublevel tableaux.
As mentioned above, for any $0<t\le1$, the \textbf{$t$-sublevel tableau of $T_{n-1}$}, which we denote by $T_{n-1}^{(t)}$, is the subtableau of $T_{n-1}$ consisting of those boxes with entries $\le t$. An essential fact that makes our main result possible is a self-similarity property that says that $T_{n-1}^{(t)}$ is distributed roughly as a scaled version of $T_k$ for $k \approx tn$. Since the number of entries in $T_{n-1}^{(t)}$ is itself random, we need to condition on its value to make this statement precise. The details are as follows.

\begin{lem} 
\label{lem:self-similarity}
Let $0<t\le 1$ and $0\le k\le n-1$. 
\begin{itemize}
\item The number $|T_{n-1}^{(t)}|$ of boxes in $T_{n-1}^{(t)}$ satisfies 
\begin{equation}
\label{eq:lln}
\frac{|T_{n-1}^{(t)}|}{n-1} \xrightarrow[n\to\infty]{\prob} t.
\end{equation}
\item Conditioned on the event $|T_{n-1}^{(t)}|=k$, the rescaled sublevel tableau $\frac{1}{t} T_{n-1}^{(t)}$ (where the rescaling means that each entry of $T_{n-1}^{(t)}$ is divided by $t$) is equal in distribution to $T_k$.
\end{itemize}
\end{lem}

\begin{proof}
Recall that the entries of $T_{n-1}$ are the i.i.d.~uniform random numbers $X_1,\dots,X_{n-1}$, so the entries of the sublevel tableau $T_{n-1}^{(t)}$ consist of the subset of the numbers $X_1,\dots,X_{n-1}$ which are $\le t$. It follows that the random variable $Z_n=|T_{n-1}^{(t)}|$ is given by
$$ Z_n =\sum_{j=1}^{n-1} 1_{\{X_j\le t\}}, $$
so the first claim follows from the law of large numbers. 

Denote by $(Y_1,\dots,Y_{Z_n})$ the sequence (of random length $Z_n$) of those $X_j$'s (for $1\le j\le n-1$) for which $X_j\le t$, in the order in which they appear. From elementary probability theory, it is easy to see that, conditioned on the event $\{Z_n=k\}$, the random variables $Y_1,\dots,Y_k$ are independent and uniformly distributed in $[0,t]$. But now observe that (still conditioning on $\{Z_n=k\}$) the sublevel tableau $T_{n-1}^{(t)}$ can be constructed as follows:
\begin{description}
\item[Standardize] \hfill \\Replace $T_{n-1}$ by a standard Young tableau $P_{n-1}$, where each entry of $P_{n-1}$ is the ranking of the corresponding entry of $T_{n-1}$ in the list $X_1,\dots,X_{n-1}$.
\item[Sublevel] \hfill \\ Take the $k$-sublevel tableau $P_{n-1}^{(k)}$ of $P_{n-1}$.
\item[Destandardize] \hfill \\Replace each entry $p$ of $P_{n-1}^{(k)}$ by the $p$th order statistic $Y^{(k)}_p$ of the sequence $Y_1,\dots,Y_k$.
\end{description}
By the remarks made at the beginning of this section, $P_{n-1}$ is a Plancherel-random standard Young tableau of order $n-1$. We now use the elementary fact that the Plancherel measures are a consistent family of probability measures, in the sense that for any $k\le m$, the $k$-sublevel tableau of a Plancherel-random tableau $Q_m$ of order $m$ is a Plancherel-random tableau of order $k$. (The case $k=m-1$ of this claim corresponds to the simple operation of removing the maximal entry of a Plancherel-random tableau; this clearly implies the general case by induction, and the fact that the claim is true in this case is a version of a well-known property of the Plancherel measures, mentioned for example in \cite[Lemma~1.25]{romik}.)
So, the tableau $P^{(k)}_{n-1}$ is distributed according to the Plancherel measure of order $k$. Finally, since (still conditioning on the event $\{Z_n=k\}$ as before) the vector of order statistics 
$(Y^{(k)}_1,\dots,Y^{(k)}_k)=(X^{(n)}_1,\dots,X^{(n)}_k)$ 
is independent of $P_{n-1}$ (and hence also of $P^{(k)}_{n-1}$) and is distributed like $t$ times a random vector distributed uniformly in $\Delta_k$, again by the remarks made above we have that $T^{(t)}_{n-1}$ is (conditionally on $\{Z_n=k\}$) equal in distribution to $t\cdot T_k$.
\end{proof}

\section{Finishing the proof}

To prove \eqref{eq:limit2}, we first reparametrize the bumping route $\bumping_{T_{n-1},\alpha}$ according to the parameter $t$ associated with the sublevel tableaux. For each $0\le \alpha< 1$, this reparametrized bumping route will now be a random function $\Phi_{n,\alpha}:[\alpha,1]\to\mathbb{N}\times\mathbb{N}$ defined by
\begin{equation} \label{eq:bumping-reparam}
\Phi_{n,\alpha}(t) = \big(b_{T_{n-1},\alpha}(m),m\big)
\end{equation}
where for each $t$ we  denote by $m$ the minimal number for which \linebreak $\big(b_{T_{n-1},\alpha}(m),m\big)$ lies outside the sublevel tableau $T_{n-1}^{(t)}$. Note that almost surely we have that 
\begin{align*}
\Phi_{n,\alpha}(\alpha)=& \big(b_{T_{n-1},\alpha}(1),1\big),\\ 
\Phi_{n,\alpha}(1)=& \big(b_{T_{n-1},\alpha}(k_{T_{n-1},\alpha}),k_{T_{n-1},\alpha}\big), 
\end{align*}
 and the range of $\Phi_{n,\alpha}$ consists of the entire bumping route $\bumping_{T_{n-1},\alpha}$.

\begin{thm}
For any $0\le \alpha<1$ and $\epsilon>0$ we have
\begin{equation} \label{eq:stronger}
\prob\left( \max_{\alpha\le t\le 1}  \left\lVert
\frac{\Phi_{n,\alpha}(t)}{\sqrt{n}} - \big(x_\alpha(t),y_\alpha(t)\big)
\right\rVert > \epsilon
\right) \xrightarrow[n\to\infty]{} 0.
\end{equation}
\end{thm}

\begin{proof}
Fix $\alpha\in [0,1)$. First, we prove the weaker statement that for any $\epsilon>0$ and $t\in[\alpha,1]$ we have
\begin{equation} \label{eq:weaker}
\prob\left( \left\lVert
\frac{\Phi_{n,\alpha}(t)}{\sqrt{n}} - \big(x_\alpha(t),y_\alpha(t)\big)
\right\rVert > \epsilon
\right) \xrightarrow[n\to\infty]{} 0.
\end{equation}
Denote $Z_n=|T_{n-1}^{(t)}|$ as before, and let $\delta>0$ be some small number (depending on $\epsilon$) whose value will be fixed shortly. We have
\begin{multline}
 \prob\left( \left\lVert
\frac{\Phi_{n,\alpha}(t)}{\sqrt{n}} - \big(x_\alpha(t),y_\alpha(t)\big)
\right\rVert > \epsilon
\right)
 \\ 
\shoveleft{\leq
\prob\left(  
\left| \frac{Z_n}{n-1} - t \right| > \delta
\right)}
 \\ 
{ \hspace{42pt} + \hspace{-5pt}
\sum_{\left|\frac{k}{n-1}-t\right| \le \delta}  \hspace{-5pt} \prob(Z_n=k)
\prob\left( \left\lVert
\frac{\Phi_{n,\alpha}(t)}{\sqrt{n}} - \big(x_\alpha(t),y_\alpha(t)\big)
\right\rVert > \epsilon \,\Big|\, Z_n=k
\right)}
 \\ 
\shoveleft{\leq
\prob\left(  
\left| \frac{Z_n}{n-1} - t \right| > \delta
\right)}
 \\ 
{ + 
\max_{\left|\frac{k}{n-1}-t\right| \le \delta} 
\prob\left( \left\lVert
\frac{\Phi_{n,\alpha}(t)}{\sqrt{n}} - \big(x_\alpha(t),y_\alpha(t)\big)
\right\rVert > \epsilon \,\Big|\, Z_n=k
\right).}
\label{eq:prob-twoparts}
\end{multline}
In the last expression, the first term tends to $0$ as $n\to\infty$, by \eqref{eq:lln}. 
Let $k=k(n)$ be the value for which the maximum of the second term is attained.
Note that the second claim of Lemma~\ref{lem:self-similarity} implies that the conditional probability in the second term can be replaced by its unconditional counterpart
\begin{multline}
\label{eq:prob-replaced}
\prob \left( \left\lVert
\frac{\Phi_{k+1,\alpha/t}(1)}{\sqrt{n}} - \big(x_\alpha(t),y_\alpha(t)\big)
\right\rVert > \epsilon 
\right)
\\ =
\prob\left( \left\lVert
\sqrt{\frac{k}{n}}\ \frac{\Phi_{k+1,\alpha/t}(1)}{\sqrt{k}} - \sqrt{t}\ \big(x_{\alpha/t}(1),y_{\alpha/t}(1)\big)
\right\rVert > \epsilon 
\right). 
\end{multline}
If we had the precise equality $k=tn$, it would immediately follow from \eqref{eq:asym-det} that this probability tends to $0$ as $n$ (and therefore also $k$) tends to $\infty$. As it is, such an equality does not hold; however, we restricted $k$ to a range such that 
$$
t-\delta \le \liminf_{n\to\infty} \frac{k}{n} \leq \limsup_{n\to\infty} \frac{k}{n} \le t+\delta.
$$
This is good enough, since it is easy to check that if $\delta$ is taken (as a function of $\epsilon$) to be a small enough positive number, then the right-hand side of \eqref{eq:prob-replaced} can be bounded from above by
\begin{multline}
\prob\left( \left\lVert
\frac{\Phi_{k+1,\alpha/t}(1)}{\sqrt{k}} - \big(x_{\alpha/t}(1),y_{\alpha/t}(1)\big)
\right\lVert > \frac{\epsilon}{2}
\right) 
 \\ +
\prob\left( \left\lVert
\frac{\Phi_{k+1,\alpha/t}(1)}{\sqrt{k}} \right\rVert > 3\sqrt{2}
\right). 
\label{eq:prob-bound2}
\end{multline}
The first probability tends to $0$ as $n\to\infty$ 
by \eqref{eq:asym-det}. The second probability is bounded by the probability that a Plancherel-random Young diagram of order $k$ has a row or column of length $\ge 3\sqrt{k}$; it is well-known that this probability decreases to $0$ at a rate that is exponential in $\sqrt{k}$ (see \cite[Lemma~1.5]{romik}). Thus, combining these observations with \eqref{eq:prob-twoparts}, \eqref{eq:prob-replaced} and the bound \eqref{eq:prob-bound2} proves \eqref{eq:weaker}.

Finally, to finish the proof we need to show that \eqref{eq:weaker} implies \eqref{eq:stronger}. This is a standard argument: first, \eqref{eq:weaker} clearly implies a version of \eqref{eq:stronger} in which the maximum is taken over finitely many values $\alpha \le t_1 < \dots < t_p \le 1$ of $t$. Second, since both the functions $t\mapsto \Phi_{n,\alpha}(t)$ and $t\mapsto \big(x_\alpha(t),y_\alpha(t)\big)$ have the property that their $x$-coordinate is weakly decreasing and their $y$-coordinate is weakly increasing, 
and since $t\mapsto \big(x_\alpha(t),y_\alpha(t)\big)$ is continuous, knowing that the bound
$$\left\lVert \frac{\Phi_{n,\alpha}(t)}{\sqrt{n}} - \big(x_\alpha(t),y_\alpha(t)\big) \right\rVert > \epsilon$$ 
holds for all values of $t$ in a finite set that is sufficiently dense in $[\alpha,1]$ ensures that the same inequality (with $\epsilon$ replaced by, say, $2\epsilon$) will hold
for all $t\in [\alpha,1]$. The details are easy and are left to the reader.
\end{proof}

\begin{proof}[Proof of \eqref{eq:limit2}]
It is now easy to derive \eqref{eq:limit2} from \eqref{eq:stronger}. The idea is that the relation between $t$ and $m$ in \eqref{eq:bumping-reparam} can be inverted, expressing the $m$th point $\big(b_{T_{n-1},\alpha}(m),m\big)$ of the bumping route as $\Phi_{n,\alpha}(t(m))$ where $t(m)$ is the minimal value $t\ge \alpha$ for which the $y$-coordinate of $\Phi_{n,\alpha}(t)$ is equal to $m$. (Note that $t(\cdot)$ also depends on $n$ and $\alpha$, but for convenience we leave this dependence implicit in our notation.) Expressing \eqref{eq:stronger} in terms of $t(m)$ gives the convergence in probability
$$
\max_{1\le m\le k_{T_{n-1},\alpha}} \left\lVert 
\frac{\big(b_{T_{n-1},\alpha}(m),m\big)}{\sqrt{n}}-\big(x_\alpha(t(m)),y_\alpha(t(m))\big)\right\rVert \xrightarrow[n\to\infty]{\prob} 0,
$$
which can be broken down into two separate convergence relations,
\begin{equation}
\max_{1\le m\le k_{T_{n-1},\alpha}} \left| y_\alpha(t(m)) - \frac{m}{\sqrt{n}} \right|  \xrightarrow[n\to\infty]{\prob} 0,  
\label{eq:first-conv}
\end{equation}
\begin{equation}
\max_{1\le m\le k_{T_{n-1},\alpha}} \left| \frac{b_{T_{n-1},\alpha}(m)}{\sqrt{n}}-x_\alpha(t(m))\right| \xrightarrow[n\to\infty]{\prob} 0.
\label{eq:second-conv} 
\end{equation}
Now observe that $z\mapsto y_\alpha^{-1}(z\wedge \kappa(\alpha))$ is a continuous function on $[0,2\kappa(\alpha)]$. Combining this with \eqref{eq:first-conv} and the fact that (by \eqref{eq:limit1}) \linebreak $k_{T_{n-1},\alpha}/\sqrt{n} \le 2\kappa(\alpha)$ with asymptotically high probability, we see that
$$
\max_{1\le m\le k_{T_{n-1},\alpha}} \left| t(m) - y_\alpha^{-1}\left(\frac{m}{\sqrt{n}} \wedge \kappa(\alpha) \right) \right| \xrightarrow[n\to\infty]{\prob} 0.
$$
Finally, this relation, together with \eqref{eq:second-conv} and the fact that $x_\alpha(\cdot)$ is continuous, implies that
$$
\max_{1\le m\le k_{T_{n-1},\alpha}} \left| \frac{b_{T_{n-1},\alpha}(m)}{\sqrt{n}}-x_\alpha\left(y_\alpha^{-1}\left(\frac{m}{\sqrt{n}} \wedge \kappa(\alpha) \right)\right)\right| \xrightarrow[n\to\infty]{\prob} 0,
$$
which is exactly \eqref{eq:limit2}.
\end{proof}

\section*{Acknowledgments}

P.\'S.'s research has been supported by a grant number SN 101/1-1 from \emph{Deutsche Forschungsgemeinschaft}.
Dan Romik's work was supported by the National Science Foundation under grant DMS-0955584 and by grant \#228524 from the Simons Foundation.

\bibliographystyle{alpha}
\bibliography{bumping}

\def\cprime{$'$}
\begin{thebibliography}{BDJ99b}

\bibitem[BDJ99a]{baik-deift-johansson}
J.~Baik, P.~Deift, and K.~Johansson.
\newblock On the distribution of the length of the longest increasing
  subsequence of random permutations.
\newblock {\em J. Amer. Math. Soc.}, 12:1119--1178, 1999.

\bibitem[BDJ99b]{baik-deift-johansson2}
J.~Baik, P.~Deift, and K.~Johansson.
\newblock On the distribution of the length of the second row of a {Y}oung
  diagram under {P}lancherel measure.
\newblock {\em Geom. Funct. Anal}, 10:702--731, 1999.

\bibitem[BOO00]{borodin-okounkov-olshanski}
A.~Borodin, A.~Okounkov, and G.~Olshanski.
\newblock Asymptotics of {P}lancherel measures for symmetric groups.
\newblock {\em J. Amer. Math. Soc.}, 13:491--515, 2000.

\bibitem[Ful97]{fulton}
W.~Fulton.
\newblock {\em Young Tableaux: With Applications to Representation theory and
  Geometry}, volume~35 of {\em London Mathematical Society Student Texts}.
\newblock Cambridge University Press, Cambridge, 1997.

\bibitem[Joh00]{johansson}
K.~Johansson.
\newblock Shape fluctuations and random matrices.
\newblock {\em Commun. Math. Phys.}, 209:437--476, 2000.

\bibitem[Knu98]{knuthvol2}
D.~E. Knuth.
\newblock {\em The Art of Computer Science, Vol. 3: Sorting and Searching, 2nd
  Ed.}
\newblock Addison-Wesley, 1998.

\bibitem[LS77]{loganshepp}
B.~F. Logan and L.~A. Shepp.
\newblock A variational problem for random {Y}oung tableaux.
\newblock {\em Adv. Math.}, 26(2):206--222, 1977.

\bibitem[Moo06]{moore}
C.~Moore.
\newblock Flows in {Y}oung diagrams, 2006.
\newblock Online resource, \url{http://tuvalu.santafe.edu/~moore/gallery.html}.

\bibitem[Rom14]{romik}
D.~Romik.
\newblock {\em The Surprising Mathematics of Longest Increasing Subsequences}.
\newblock Cambridge University Press, 2014.
\newblock To appear; available at \newline
  \mbox{\url{http://www.math.ucdavis.edu/~romik/book}}.

\bibitem[R{\'S}14]{romik-sniady}
D.~Romik and P.~{\'S}niady.
\newblock Jeu de taquin dynamics on infinite {Y}oung tableaux and second class
  particles.
\newblock {\em Ann. Probab.}, 2014.
\newblock To appear.

\bibitem[Sta99]{stanleyvol2}
R.~P. Stanley.
\newblock {\em Enumerative combinatorics. {V}ol. 2}, volume~62 of {\em
  Cambridge Studies in Advanced Mathematics}.
\newblock Cambridge University Press, Cambridge, 1999.

\bibitem[VK77]{vershikkerov1}
A.~M. Vershik and S.~V. Kerov.
\newblock Asymptotics of the {P}lancherel measure of the symmetric group and
  the limit form of {Y}oung tableaux.
\newblock {\em Soviet Math. Dokl.}, 18:527--531, 1977.

\bibitem[VK85]{vershikkerov2}
A.~M. Vershik and S.~V. Kerov.
\newblock Asymptotic of the largest and the typical dimensions of irreducible
  representations of a symmetric group.
\newblock {\em Functional Anal. Appl.}, 19(1):21--31, 1985.

\end{thebibliography}

\end{document}